%% file: limo10english.tex
\documentclass[a4paper,11pt]{article}
\usepackage{amsmath, amssymb, amsthm}
\usepackage{pstricks, pst-node}
\SpecialCoor
\usepackage[english]{babel}
\newtheorem{theorem}{Theorem}
\newtheorem{lemma}[theorem]{Lemma}
\newtheorem{corollary}[theorem]{Corollary}

\newcommand{\figuurbase}[3]{
 \begin{figure}[ht]
 \begin{center}
   \fbox{\parbox{12cm}{
     \centering
     #1
     \caption{#2}
     \label{#3}
   }}
 \end{center}
 \end{figure}
}

\newcommand{\Z}{\mathbb{Z}}

\title{A combinatorial solution to LIMO 2010 question 10}
\author{Stijn Vermeeren\footnote{Thanks to my brother Mats Vermeeren, who was a member of the team from the Katholieke Universiteit Leuven that won the LIMO 2010. Through him I was introduced to this problem.} \ (University of Leeds)}
\date{\today}

\begin{document}

\maketitle

\begin{abstract}
Problem 10 of the Landelijke Interuniversitaire Mathematische Olympiade 2010 asks for a proof that all matrices in a certain family are nilpotent. Both model solutions prove this using the Cayley-Hamilton theorem. I give a purely combinatorial proof.
\end{abstract}

The LIMO (Landelijke Interuniversitaire Mathematische Olympiade) is a mathematics competition that is organized annually in the Netherlands \cite{limosite}. Mathematics students from the Netherlands and from Flanders participate in teams. The 2010 competition was held on May 28 in Utrecht. Problem number 10 was proposed by Jaap Top from the University of Groningen and reads as follows:

\begin{quotation}
Let $m$ be a positive integer and $n = 2^m - 1$. The $n \times n$ matrix $A = (a_{i,j})$ over $\Z / 2\Z$ is given by $a_{i,j} = 1$ if $|i - j| = 1$ and $a_{i,j} = 0$ otherwise. Show that $A$ is nilpotent.

\textbf{Remark:} \emph{A matrix $A$ is nilpotent if some power of $A$ equals the zero matrix.}
\end{quotation}

Two model solutions were given \cite{uitwerkingen}, proving in particular that \mbox{$A^n = \mathbf{0}$}. Both model solutions use the Cayley-Hamilton theorem: every square matrix over a commutative ring is a \emph{root} (in the obvious sense) of its own characteristic polynomial. The rest of the solution is to ingeniously calculate that characteristic polynomial, which turns out to be equal to $\lambda^n$.

However, it is striking that $A$ is the adjacency matrix of the path graph $P_n$. And it is well known that the entries of the $k$'th power of an adjacency matrix correspond to the number of walks of length $k$ between two particular vertices of the graph. This suggests a combinatorial approach to the problem.

\section{Background from graph theory}

A \textbf{graph} $G$ consists of a finite set $V$ of \textbf{vertices} and a set $E$ of \textbf{edges}. Every edge is a (unordered) pair of vertices. If $\{x,y\}$ is an edge of $G$, then we say that the vertices $x$ and $y$ are \textbf{connected} by this edge. (Note that a vertex is \emph{not} connected to itself.) A \textbf{walk} in $G$ is a finite sequence $x_0, x_1, \ldots, x_k$ of vertices, such that every vertex in the sequence is connected to the next one. Vertex $x_0$ is the \textbf{start}, vertex $x_n$ is the \textbf{end} and $k$ is the \textbf{length} of the walk.

Let $n$ be a positive integer. The \textbf{path graph} $P_n$ is the graph with vertex set $\{1,2,\cdots, n\}$ and in which two integers are connected if and only if they are consecutive. The path graph $P_7$ for example, can be depicted as follows:

  \figuurbase{ \input{p7} }{The path graph $P_7$.}{p7}

Let $G$ be a graph with $n$ vertices. We assume from now on that the vertices are always numbered from $1$ to $n$. The \textbf{adjacency matrix} of $G$ is the $n \times n$ matrix A with 
\[
a_{x,y} =
	\begin{cases}
		1 & \mbox{if the vertices $x$ are $y$ connected,} \\
		0 & \mbox{otherwise.}
	\end{cases}
\]

For example, the path graph $P_7$ has adjacency matrix
\[
\begin{pmatrix}
	0 & 1 & 0 & 0 & 0 & 0 & 0 \\
	1 & 0 & 1 & 0 & 0 & 0 & 0 \\
	0 & 1 & 0 & 1 & 0 & 0 & 0 \\
	0 & 0 & 1 & 0 & 1 & 0 & 0 \\
	0 & 0 & 0 & 1 & 0 & 1 & 0 \\
	0 & 0 & 0 & 0 & 1 & 0 & 1 \\
	0 & 0 & 0 & 0 & 0 & 1 & 0
\end{pmatrix}.
\]

\begin{lemma}
\label{lemma}
Let $G$ be a graph with $n$ vertices and adjacency matrix $A$. Let $k$ be a positive integer and let $B = A^k$. Then $b_{x,y}$ equals the number of walks of length $k$ that start in vertex $x$ and end in vertex $y$.
\end{lemma}

\begin{proof}
We prove be induction on $k$. The case $k=1$ is satisfied by definition of adjacency matrix. Indeed, the unique walk of length $1$ from vertex $x$ to vertex $y$ exists only if $x$ and $y$ are connected.

Suppose now that the lemma is correct for $k=l$. Let $C = A^l$ and consider $B = A^{l+1} = C \cdot A$. By definition of matrix multiplication we have
\[ b_{x,y} = \sum_{z=1}^n c_{x,z} \cdot a_{z,y}. \]
Here $c_{x,z}$ is the number of walks of length $l$ from $x$ to $z$ by induction hypothesis. Moreover $a_{z,y}$ equals $1$ if $z$ are $y$ connected, and $0$ otherwise, by definition of adjacency matrix. So $b_{x,y}$ is the sum of the number of walks of length $l$ from $x$ to $z$ over all vertices $z$ that are connected to $y$. This is the number of walks of length $l+1$ from $x$ to $y$, summed over the penultimate vertex of the walk. Hence the lemma is also satisfied for $k=l+1$.

By induction, the lemma is proven for all positive integers $k$.
\end{proof}

\section{A combinatorial solution}

\begin{theorem}
Let $m$ be a positive integer, $n = 2^m - 1$ and $k$ an integer with $k \geq n$. Take any two vertices $x$ and $y$ of the path graph $P_n$. The number of walks in $P_n$ of length $k$ from $x$ to $y$ is even.
\end{theorem}

\begin{proof}
We give a proof by induction on $m$.

The case $m=1$ (and hence $n=1$) is trivial, as there are no walks of positive length in $P_1$. 

So suppose the theorem is true for $m=l$ and consider $n=2^{l+1}-1$ and $k \geq 2^{l+1}-1$. Let $x$ and $y$ be two vertices of $P_n$. We partition the walks of length $k$ from $x$ to $y$ in three classes:
\begin{description}
	\item[\textbf{Class 1:}] the walks that never visit $2^l$.
	\item[\textbf{Class 2:}] the walks that visit $2^l$ exactly once.
	\item[\textbf{Class 3:}] the walks that visit $2^l$ at least twice.
\end{description}
We will prove that each of these classes contains an even number of walks, which immediately implies that the theorem is also true for $m=l+1$.

Essential in applying the induction hypothesis is the fact that the two subgraphs of $P_{2^{l+1}-1}$ which are to the left and to the right of vertex $2^l$, are isomorphic to $P_{2^l-1}$. Hence, the walks in $P_{2^{l+1}-1}$ that are completely on the left (or on the right) of $2^l$ correspond to the walks in $P_{2^{l}-1}$.

The three cases are summarized in Figure \ref{drieklassen}.

  \figuurbase{ \input{drieklassenenglish} }{An illustration of how the parity of the number of walks in each class is counted. Here we have $m=3$, $n=7$, $k=7$, $x=3$ and $y=2$.}{drieklassen}

\begin{description}
	\item[\textbf{Class 1:}] In case $x$ and $y$ are on different sides of $2^l$, there are no walks in this class. If $x$ and $y$ are on the same side of $2^l$, then the walks in this class are exactly the walks of length $k$ from $x$ to $y$ in the subgraph of $P_{2^{l+1}-1}$ which is on that side of $2^l$. This subgraph is isomorphic to $P_{2^{l}-1}$, so by induction hypothesis the number of these walks is even.
	\item[\textbf{Class 2:}] We claim that for all $i \in \{0,1,2,\cdots,k\}$, there's an even number of walks in this class that visit $2^l$ after exactly $i$ steps. A such walk
\[ x, a_1, \cdots, a_{i-1}, 2^l, a_{i+1}, \cdots, a_{k-1},y \]
is determined by two subwalks, each completely at one side of $2^l$: one walk of length $i-1$ from $x$ to $a_{i_1}$ and one walk of length $k-i-1$ from $a_{i+1}$ to $y$. (If $i=0$ or $i=k$, the walk is determined by \emph{one} such subwalk of length $k-1$.) Since \[k \geq 2^{l+1} - 1\] we have either \[ i-1 \geq 2^l - 1 \] or \[ k-i-1 \geq 2^l - 1.\] By induction hypothesis, we can choose one of these subwalks in an even number of ways. Therefore the total number of such walks is also even.
	\item[\textbf{Class 3:}] In the walks in this class, we can \emph{reflect} the subwalk between the first to occurrences of $2^l$. To reflect here means: each step to the \emph{right} (i.e. from $i$ to $i+1$) is replaced by a step to the \emph{left} (i.e. from $i$ to $i-1$) and the other way around. Like this we obtain a different walk in this class. Furthermore, on applying this operation twice, we return to the original walk. Hence we can partition the walks in pairs (a walk together with its \emph{mirror image}), so the number of walks in this class is even.
\end{description}
By induction the theorem is valid for all positive integers $m$, as required.
\end{proof}

Note that the bound $k \geq n$ is optimal. Indeed, there is only one walk in $P_n$ of length $n-1$ from vertex $1$ to vertex $n$.

Also note that it is not possible to simply consider the first vertex $x$ with maximal $2$-exponent (i.e. the exponent of $2$ in the prime factorization) which is visited at least twice in the walk, and to reflect the subwalk between the first two occurrences of $x$. Like that, $x$ remains the first vertex with maximal $2$-exponent that is visited twice in the new walk, so that reflecting twice takes us back to the original walk. But the reflection is not always well-defined, as is illustrated in Figure \ref{tegenvb}. The above proof avoids this problem by splitting the walk into two subwalks where the walk visits $4$, whereby the subwalk that Figuur \ref{tegenvb} tried to reflect is broken in two (see the middle picture in Figure \ref{drieklassen}).

  \figuurbase{ \input{tegenvb} }{How a naive method goes awry. The middle picture in Figuur \ref{drieklassen} illustrate how this example can be treated correctly.}{tegenvb}

A solution to the LIMO question is now an immediate corollary.

\begin{corollary}
Let $m$ be a positive integer and $n = 2^m - 1$. The $n \times n$ matrix $A = (a_{i,j})$ over $\Z / 2\Z$ is given by $a_{i,j} = 1$ if $|i - j| = 1$ and $a_{i,j} = 0$ otherwise. Then $A^n = \mathbf{0}$.
\end{corollary}

\begin{proof}
The matrix $A$ is the adjacency matrix of $P_n$. Write $B = A^n$. By Lemma \ref{lemma}, and because we are working over $\Z / 2\Z$, $b_{x,y}$ is the parity of the number of walks of length $n$ from $x$ to $y$ in $P_n$. This number is even by the above theorem, so $b_{x,y} = 0$ for all $x,y$. Consequently $A^n = B = \mathbf{0}$, as required.
\end{proof}

\bibliographystyle{plain}
\bibliography{limo10english}

\end{document}

%% file: p7.tex
\psset{unit=0.8cm}
\begin{pspicture}(0.7,0)(8,2)

{
\cnode(1,1){2.3pt}{n1}
\uput{4.5pt}[u](n1){$1$}
\cnode(2,1){2.3pt}{n2}
\uput{4.5pt}[u](n2){$2$}
\cnode(3,1){2.3pt}{n3}
\uput{4.5pt}[u](n3){$3$}
\cnode(4,1){2.3pt}{n4}
\uput{4.5pt}[u](n4){$4$}
\cnode(5,1){2.3pt}{n5}
\uput{4.5pt}[u](n5){$5$}
\cnode(6,1){2.3pt}{n6}
\uput{4.5pt}[u](n6){$6$}
\cnode(7,1){2.3pt}{n7}
\uput{4.5pt}[u](n7){$7$}
}

\ncline{n1}{n2}
\ncline{n2}{n3}
\ncline{n3}{n4}
\ncline{n4}{n5}
\ncline{n5}{n6}
\ncline{n6}{n7}

\end{pspicture}

%% file: drieklassenenglish.tex
\psset{unit=0.5cm}
\begin{pspicture}(0.7,-2.5)(22,4)

\rput(0,0){

	\rput(4,3){\textbf{Class 1}}

	\cnode(1,1){2.3pt}{n1}
	\uput{4.5pt}[u](n1){$1$}
	\cnode(2,1){2.3pt}{n2}
	\uput{4.5pt}[u](n2){$2$}
	\cnode(3,1){2.3pt}{n3}
	\uput{4.5pt}[u](n3){$3$}
	\cnode(4,1){2.3pt}{n4}
	\uput{4.5pt}[u](n4){$4$}
	\cnode(5,1){2.3pt}{n5}
	\uput{4.5pt}[u](n5){$5$}
	\cnode(6,1){2.3pt}{n6}
	\uput{4.5pt}[u](n6){$6$}
	\cnode(7,1){2.3pt}{n7}
	\uput{4.5pt}[u](n7){$7$}

	\ncline{n1}{n2}
	\ncline{n2}{n3}
	\ncline{n3}{n4}
	\ncline{n4}{n5}
	\ncline{n5}{n6}
	\ncline{n6}{n7}

	\rput(n1){\psline[linestyle=dotted](0,-3.5)}
	\rput(n2){\psline[linestyle=dotted](0,-3.5)}
	\rput(n3){\psline[linestyle=dotted](0,-3.5)}
	\rput(n4){\psline[linestyle=dotted,linewidth=1.6pt](0,-3.5)}
	\rput(n5){\psline[linestyle=dotted](0,-3.5)}
	\rput(n7){\psline[linestyle=dotted](0,-3.5)}
	\rput(n6){\psline[linestyle=dotted](0,-3.5)}

	\pnode(3,0.33){p1}
	\pnode(2,0){p2}
	\pnode(3,-0.33){p3}
	\pnode(2,-0.66){p4}
	\pnode(1,-1){p5}
	\pnode(2,-1.33){p6}
	\pnode(3,-1.66){p7}
	\pnode(2,-2){p8}

	\ncline[arrowsize=2.7pt,nodesep=1pt]{->}{p1}{p2}
	\ncline[arrowsize=2.7pt,nodesep=1pt]{->}{p2}{p3}
	\ncline[arrowsize=2.7pt,nodesep=1pt]{->}{p3}{p4}
	\ncline[arrowsize=2.7pt,nodesep=1pt]{->}{p4}{p5}
	\ncline[arrowsize=2.7pt,nodesep=1pt]{->}{p5}{p6}
	\ncline[arrowsize=2.7pt,nodesep=1pt]{->}{p6}{p7}
	\ncline[arrowsize=2.7pt,nodesep=1pt]{->}{p7}{p8}

	\psframe(0.8,0.53)(3.2,-2.2)
	\rput*(2.4,-2.6){\tiny \bf Induction hypothesis}
	
}

\rput(7.5,0){

	\rput(4,3){\textbf{Class 2}}
	
	\cnode(1,1){2.3pt}{n1}
	\uput{4.5pt}[u](n1){$1$}
	\cnode(2,1){2.3pt}{n2}
	\uput{4.5pt}[u](n2){$2$}
	\cnode(3,1){2.3pt}{n3}
	\uput{4.5pt}[u](n3){$3$}
	\cnode(4,1){2.3pt}{n4}
	\uput{4.5pt}[u](n4){$4$}
	\cnode(5,1){2.3pt}{n5}
	\uput{4.5pt}[u](n5){$5$}
	\cnode(6,1){2.3pt}{n6}
	\uput{4.5pt}[u](n6){$6$}
	\cnode(7,1){2.3pt}{n7}
	\uput{4.5pt}[u](n7){$7$}

	\ncline{n1}{n2}
	\ncline{n2}{n3}
	\ncline{n3}{n4}
	\ncline{n4}{n5}
	\ncline{n5}{n6}
	\ncline{n6}{n7}

	\rput(n1){\psline[linestyle=dotted](0,-3.5)}
	\rput(n2){\psline[linestyle=dotted](0,-3.5)}
	\rput(n3){\psline[linestyle=dotted](0,-3.5)}
	\rput(n4){\psline[linestyle=dotted,linewidth=1.6pt](0,-3.5)}
	\rput(n5){\psline[linestyle=dotted](0,-3.5)}
	\rput(n7){\psline[linestyle=dotted](0,-3.5)}
	\rput(n6){\psline[linestyle=dotted](0,-3.5)}

	\pnode(3,0.33){p1}
	\pnode(2,0){p2}
	\pnode(3,-0.33){p3}
	\pnode(4,-0.66){p4}
	\pnode(3,-1){p5}
	\pnode(2,-1.33){p6}
	\pnode(1,-1.66){p7}
	\pnode(2,-2){p8}

	\ncline[arrowsize=2.7pt,nodesep=1pt]{->}{p1}{p2}
	\ncline[arrowsize=2.7pt,nodesep=1pt]{->}{p2}{p3}
	\ncline[arrowsize=2.7pt,nodesep=1pt]{->}{p3}{p4}
	\ncline[arrowsize=2.7pt,nodesep=1pt]{->}{p4}{p5}
	\ncline[arrowsize=2.7pt,nodesep=1pt]{->}{p5}{p6}
	\ncline[arrowsize=2.7pt,nodesep=1pt]{->}{p6}{p7}
	\ncline[arrowsize=2.7pt,nodesep=1pt]{->}{p7}{p8}

	\psframe(0.8,-.8)(3.2,-2.2)
	\rput*(2.4,-2.6){\tiny \bf Induction hypothesis}

}

\rput(15,0){

	\rput(4,3){\textbf{Class 3}}

	\cnode(1,1){2.3pt}{n1}
	\uput{4.5pt}[u](n1){$1$}
	\cnode(2,1){2.3pt}{n2}
	\uput{4.5pt}[u](n2){$2$}
	\cnode(3,1){2.3pt}{n3}
	\uput{4.5pt}[u](n3){$3$}
	\cnode(4,1){2.3pt}{n4}
	\uput{4.5pt}[u](n4){$4$}
	\cnode(5,1){2.3pt}{n5}
	\uput{4.5pt}[u](n5){$5$}
	\cnode(6,1){2.3pt}{n6}
	\uput{4.5pt}[u](n6){$6$}
	\cnode(7,1){2.3pt}{n7}
	\uput{4.5pt}[u](n7){$7$}

	\ncline{n1}{n2}
	\ncline{n2}{n3}
	\ncline{n3}{n4}
	\ncline{n4}{n5}
	\ncline{n5}{n6}
	\ncline{n6}{n7}

	\psline[linestyle=dotted](1,0.7)(1,-2.5)
	\psline[linestyle=dotted](2,0.7)(2,-2.5)
	\psline[linestyle=dotted](3,0.7)(3,-2.5)
	\psline[linestyle=dotted,linewidth=1.6pt](4,0.7)(4,-2.5)
	\psline[linestyle=dotted](5,0.7)(5,-2.5)
	\psline[linestyle=dotted](6,0.7)(6,-2.5)
	\psline[linestyle=dotted](7,0.7)(7,-2.5)
	
	\pnode(3,0.33){p1}
	\pnode(4,0){p2}
	\pnode(5,-0.33){p3}
	\pnode(3,-0.33){q3}
	\pnode(4,-0.66){p4}
	\pnode(5,-1){p5}
	\pnode(4,-1.33){p6}
	\pnode(3,-1.66){p7}
	\pnode(2,-2){p8}

	\ncline[arrowsize=2.7pt,nodesep=1pt]{->}{p1}{p2}
	\ncline[arrowsize=2.7pt,nodesep=1pt,linestyle=dashed,dash=1.3pt 1.3pt]{->}{p2}{p3}
	\ncline[arrowsize=2.7pt,nodesep=1pt,linestyle=dashed,dash=1.3pt 1.3pt]{->}{p2}{q3}
	\ncline[arrowsize=2.7pt,nodesep=1pt,linestyle=dashed,dash=1.3pt 1.3pt]{->}{p3}{p4}
	\ncline[arrowsize=2.7pt,nodesep=1pt,linestyle=dashed,dash=1.3pt 1.3pt]{->}{q3}{p4}
	\ncline[arrowsize=2.7pt,nodesep=1pt]{->}{p4}{p5}
	\ncline[arrowsize=2.7pt,nodesep=1pt]{->}{p5}{p6}
	\ncline[arrowsize=2.7pt,nodesep=1pt]{->}{p6}{p7}
	\ncline[arrowsize=2.7pt,nodesep=1pt]{->}{p7}{p8}
	
	\uput*{3.0pt}[r](5,-.33){\tiny \bf Reflection}
}

\end{pspicture}

%% file: tegenvb.tex
\psset{unit=0.8cm}
\begin{pspicture}(-.5,-2.5)(8,2)

{
\cnode(1,1){2.3pt}{n1}
\uput{4.5pt}[u](n1){$1$}
\cnode(2,1){2.3pt}{n2}
\uput{4.5pt}[u](n2){$2$}
\cnode(3,1){2.3pt}{n3}
\uput{4.5pt}[u](n3){$3$}
\cnode(4,1){2.3pt}{n4}
\uput{4.5pt}[u](n4){$4$}
\cnode(5,1){2.3pt}{n5}
\uput{4.5pt}[u](n5){$5$}
\cnode(6,1){2.3pt}{n6}
\uput{4.5pt}[u](n6){$6$}
\cnode(7,1){2.3pt}{n7}
\uput{4.5pt}[u](n7){$7$}
}

\ncline{n1}{n2}
\ncline{n2}{n3}
\ncline{n3}{n4}
\ncline{n4}{n5}
\ncline{n5}{n6}
\ncline{n6}{n7}

\rput(n1){\psline[linestyle=dotted](0,-3.5)}
\rput(n2){\psline[linestyle=dotted,linewidth=1.6pt](0,-3.5)}
\rput(n3){\psline[linestyle=dotted](0,-3.5)}
\rput(n4){\psline[linestyle=dotted](0,-3.5)}
\rput(n5){\psline[linestyle=dotted](0,-3.5)}
\rput(n7){\psline[linestyle=dotted](0,-3.5)}
\rput(n6){\psline[linestyle=dotted](0,-3.5)}

\pnode(3,0.33){p1}
\pnode(2,0){p2}
\pnode(3,-0.33){p3}
\pnode(1,-0.33){q3}
\pnode(4,-0.66){p4}
\pnode(0,-0.66){q4}
\pnode(3,-1){p5}
\pnode(1,-1){q5}
\pnode(2,-1.33){p6}
\pnode(1,-1.66){p7}
\pnode(2,-2){p8}

\ncline[arrowsize=2.7pt,nodesep=1pt]{->}{p1}{p2}
\ncline[arrowsize=2.7pt,nodesep=1pt]{->}{p2}{p3}
\ncline[arrowsize=2.7pt,nodesep=1pt,linestyle=dashed,dash=1.3pt 1.3pt]{->}{p2}{q3}
\ncline[arrowsize=2.7pt,nodesep=1pt]{->}{p3}{p4}
\ncline[arrowsize=2.7pt,nodesep=1pt,linestyle=dashed,dash=1.3pt 1.3pt]{->}{q3}{q4}
\ncline[arrowsize=2.7pt,nodesep=1pt]{->}{p4}{p5}
\ncline[arrowsize=2.7pt,nodesep=1pt,linestyle=dashed,dash=1.3pt 1.3pt]{->}{q4}{q5}
\ncline[arrowsize=2.7pt,nodesep=1pt]{->}{p5}{p6}
\ncline[arrowsize=2.7pt,nodesep=1pt,linestyle=dashed,dash=1.3pt 1.3pt]{->}{q5}{p6}
\ncline[arrowsize=2.7pt,nodesep=1pt]{->}{p6}{p7}
\ncline[arrowsize=2.7pt,nodesep=1pt]{->}{p7}{p8}

\psline[linewidth=2pt](.1,-0.26)(.7,-1.06)
\psline[linewidth=2pt](.7,-0.26)(.1,-1.06)

\end{pspicture}